\documentclass[secnum,leqno]{rims-bessatsu}
\usepackage{amssymb, amsmath}
\usepackage{graphics}
\usepackage[dvips]{graphicx}
\usepackage{eepic}
\usepackage{epic}
\newtheorem{thm}{Theorem}[section]

\theoremstyle{definition}

\newtheorem{exa}[thm]{Example}
\theoremstyle{remark}
\newtheorem*{rem}{Remark}
\newcommand\DN{\newcommand}\newcommand\DR{\renewcommand}

\DN\sx{\theta_{-x} (\mathsf{s}-\delta_x)}
\DN\si{\theta_{-\lab _i (\mathsf{s})} (\mathsf{s}-\delta_{\lab _i (\mathsf{s})} )}

\DN\lref[1]{Lemma~\ref{#1}}
\DN\tref[1]{Theorem~\ref{#1}}
\DN\pref[1]{Proposition~\ref{#1}}
\DN\sref[1]{Section~\ref{#1}}
\DN\ssref[1]{Subsection~\ref{#1}}
\DN\dref[1]{Definition~\ref{#1}}
\DN\rref[1]{Remark~\ref{#1}} 
\DN\corref[1]{Corollary~\ref{#1}}
\DN\eref[1]{Example~\ref{#1}}
\numberwithin{equation}{section}
\newcounter{Const} 

\DN\Ct{\refstepcounter{Const}c_{\theConst}}
\numberwithin{Const}{section}
\DN\cref[1]{c_{\ref{#1}}}

\DN\R{\mathbb{R}}
\DN\N{\mathbb{N}}
\DN\Q{\mathbb{Q}}
\DN\C{\mathbb{C}}
\DN\Z{\mathbb{Z}}

\DN\map[3]{#1\!:\!#2\!\to\!#3}
\DN\ot{ \otimes } 
\DN\ts{ \times }

\DN\limi[1]{\lim_{#1\to\infty}} 	
\DN\limz[1]{\lim_{#1\to0}}
\DN\limsupi[1]{\limsup_{#1\to\infty}} 	
\DN\limsupz[1]{\limsup_{#1\to0}}
\DN\liminfi[1]{\liminf_{#1\to\infty}} 	
\DN\liminfz[1]{\liminf_{#1\to0}}

\DN\sumii[1]{\sum_{#1=1}^{\infty}}
\DN\sumi[1]{\sum_{#1=0}^{\infty}}

\DN\PD[2]{\frac{\partial#1}{\partial#2}}
\DN\half{\frac{1}{2}}
\DN\Rd{\R ^d}

\DN\bs{\bigskip}
\DN\ms{\smallskip}

\usepackage{stackrel} 

\newtheorem{lem}[thm]{Lemma}

\DN\CR{C([0,\infty);\R )}
\DN\CRd{C([0,\infty);\Rd )}

\DN\jZi{j\in\Z,\, j\not= i }
\DN\iZs{i\in\Z ^*}
\DN\iZ{i\in\Z }
\DN\Emu{\mathcal{E}^{\mu } }
\DN\Emuz{\mathcal{E}^{\muz } }
\DN\Lmu{L^2(\mathsf{S},\mu )}
\DN\Lmuone{L^2(\mathsf{S},\mu ^{[1]})}
\DN\Lmuz{L^2(\mathsf{S},\muz )}
\DN\D{D^{\mathrm{sft}}}
\DN\ee{1}
\DN\ve{\varepsilon}
\DN\X{X}
\DN\selfmu{\alpha [\mu ]}

\DN\Lmz{L^{2}(\muz )}
\DN\Lmg{L^{2}(\mug )}
\DN\Lmgz{L^{2}(\mugz )}

\DN\Lm{L^{2}(\mu )}
\DN\Lmgone{L^{2}(\mugone )}
\DN\Llocmgone{L^{2}_{\mathrm{loc}}(\mugone )}

\DN\Lmuk{L^{2}(\muk )}
\DN\Lmone{L^{2}(\muone )}
\DN\Llocp{L_{\mathrm{loc}}^{p}}
\DN\Llocq{L_{\mathrm{loc}}^{q}}
\DN\Lloctwo{L_{\mathrm{loc}}^{2}}
\DN\Llocone{L_{\mathrm{loc}}^{1}}

\DN\Rs{\mathbf{R}^2}
\DN\RSR{\Rs \ts \SSS \ts \Rs }
\DN\Dsft{D^{\mathrm{sft}} }
\DN\Dsftone{\Dsft _1}
\DN\Dsfttwo{\Dsft _2}

\DN\XY{X\Y}
\DN\DXY{\mathbb{D}_{\XY }}
\DN\DY{\mathbb{D}_{\Y }}
\DN\nablastar{\nabla_{x^{\star}}}
\DN\nablaYZ{\nabla _{\xstar + \mathrm{s ft}}}
\DN\Dstar {\mathbb{D}_{\star }}
\DN\DXYstar{\mathbb{D}_{\XY \star }}

\DN\Ce{\gamma }
\DN\Cet{\tilde{\Ce }}
\DN\Ceti{\tilde{\Ce }_i}
\DN\xstar{x^{\star }}
\DN\Ystar{Y^{\star }}
\DN\hatYstar{\hat{Y}^{\star }}

\DN\CY{C^{\infty} (\Rd )\ot\dY } 
\DN\dione{C^{\infty}_{0} (\SS )\ot \di }

\DN\dom{\mathcal{D}}
\DN\dz{\dom _{\circ }}
\DN\dzr{\dom _{\circ , r}} 
\DN\dzq{\dom _{\circ , q}} 
\DN\dnuik{C_0^{\infty}(\Sk )\ot\di }

\DN\di{C_0^{\infty}(\R )}

\DN\dY{\mathcal{D}_{\Y }}
\DN\dYz{\mathcal{D}_{\Y \circ }}
\DN\dtY{\dt _{\Y }}
\DN\dYr{\mathcal{D}_{\Y ,r}}
\DN\dYtwo{\dtY ^{2}}
\DN\dt{\tilde{\mathcal{D}}}
\DN\dtYz{\dt _{\Y \circ }}
\DN\dYtr{\tilde{\mathcal{D}}_{\Y ,r}}
\DN\dYtrone{\tilde{\mathcal{D}}_{\Y ,r}^1}
\DN\dYtrtwo{\tilde{\mathcal{D}}_{\Y ,r}^2}

\DN\dYtone{\dtY ^1}
\DN\dYttwo{\dtY ^2}
\DN\dYttwoR{\dt _{\Y ,r} ^2}

\DN\dYL{\mathcal{D}_{\Y }^{\mathrm{L}}}
\DN\dYLr{\mathcal{D}_{\Y ,r}^{\mathrm{L}}}
\DN\dYLtwo{\dtY ^{\mathrm{L},2}}
\DN\dYLt{\dtY ^{\mathrm{L}}}

\DN\dYLtperp{\dtY ^{2,\perp }}
\DN\dYLtr{\tilde{\mathcal{D}}_{\Y ,r}^{\mathrm{L}}}
\DN\dYLtrone{\tilde{\mathcal{D}}_{\Y ,r}^{\mathrm{L},1}}
\DN\dYLtrtwo{\tilde{\mathcal{D}}_{\Y ,r}^{\mathrm{L},2}}

\DN\dYLtq{\tilde{\mathcal{D}}_{\Y ,q}^{\mathrm{L}}}
\DN\dYLtqone{\tilde{\mathcal{D}}_{\Y ,q}^{\mathrm{L},1}}
\DN\dYLtqtwo{\tilde{\mathcal{D}}_{\Y ,q}^{\mathrm{L},2}}

\DN\dYLtone{\dtY ^{\mathrm{L},1}}
\DN\dYLttwo{\dtY ^{\mathrm{L},2}}

\DN\dXY{\mathcal{D}_{\XY }}
\DN\dXYloc{\mathcal{D}_{\XY ,\mathrm{loc}}}
\DN\dXYstar{\mathcal{D}_{\XY \star }}
\DN\dXYstarz{\mathcal{D}_{\circ}(\RSR ) } 
\DN\dYQ{\widetilde{\mathcal{D}}_{\Y }}
\DN\dYQz{\widetilde{\mathcal{D}}_{\Y \circ }}
\DN\dstar{\mathcal{D}_{\star}}
\DN\dYstarz{\mathcal{D}_{\circ}(\SSS \ts \Rs ) }

\DN\dzYZ{\dz (\SSS \ts \R  )}

\DN\EXY{\mathcal{E}_{\XY }}
\DN\EXYone{\mathcal{E}_{\XY }^1}
\DN\EXYtwo{\mathcal{E}_{\XY }^2}
\DN\EXYstar{\mathcal{E}_{\XY \star}}

\DN\Y{\mathsf{Y}}
\DN\EY{\mathcal{E}_{\Y }}
\DN\EYone{\mathcal{E}_{\Y }^{1}}
\DN\EYtwo{\mathcal{E}_{\Y }^{2}}
\DN\EYt{\tilde{\mathcal{E}}_{\Y }}
\DN\EYtone{\tilde{\mathcal{E}}_{\Y }^{1}}
\DN\EYttwo{\tilde{\mathcal{E}}_{\Y }^{2}}
\DN\EYtrperp{\EYtr ^{2,\perp }}
\DN\EYtrgamma{\EYtr ^{2,\gamma }}
\DN\EYtperp{\EYt ^{2,\perp }}
\DN\EYtr{\tilde{\mathcal{E}}_{\Y ,r}}
\DN\EYtrone{\EYtr ^{1}}
\DN\EYtrtwo{\EYtr ^{2}}

\DN\EYL{\mathcal{E}_{\Y }^{\mathrm{L}}}
\DN\EYLone{\mathcal{E}_{\Y }^{\mathrm{L},1}}
\DN\EYLtwo{\mathcal{E}_{\Y }^{\mathrm{L},2}}

\DN\EYLr{\mathcal{E}_{\Y ,r}^{\mathrm{L}}}
\DN\EYLrone{\mathcal{E}_{\Y ,r}^{\mathrm{L},1}}
\DN\EYLrtwo{\mathcal{E}_{\Y ,r}^{\mathrm{L},2}}

\DN\EYLt{\tilde{\mathcal{E}}_{\Y }^{\mathrm{L}}}
\DN\EYLtone{\tilde{\mathcal{E}}_{\Y }^{\mathrm{L},1}}
\DN\EYLttwo{\tilde{\mathcal{E}}_{\Y }^{\mathrm{L},2}}

\DN\EYLtr{\EYtr ^{\mathrm{L}}}
\DN\EYLtrone{\EYtr ^{\mathrm{L},1}}
\DN\EYLtrtwo{\EYtr ^{\mathrm{L},2}}

\DN\EYLtq{\tilde{\mathcal{E}}_{\Y ,q}^{\mathrm{L}}}
\DN\EYLtqone{\tilde{\mathcal{E}}_{\Y ,q}^{\mathrm{L},1}}
\DN\EYLtqtwo{\tilde{\mathcal{E}}_{\Y ,q}^{\mathrm{L},2}}

\DN\Estar{\mathcal{E}_{\star }}

\DN\PY{\PP ^{\Y }}
\DN\PZ{\PP ^{\Z }}
\DN\PXY{\PP ^{\XY }}

\DN\ulab{\mathfrak{u} }
\DN\kpath{\ulab _{\mathrm{path}}}
\DN\kkpath{\ulab _{\mathrm{k,path}}}
\DN\lab{\mathfrak{l}}
\DN\labi{\mathfrak{l}_i}
\DN\labz{\mathfrak{l}_0}
\DN\hhh{\mathfrak{h}}

\DN\fff{\mathbf{f}}
\DR\ggg{\mathbf{g}}
\DR\gg{\mathsf{g}}
\DN\ff{\mathsf{f}}
\DN\f{\mathsf{F}} 
\DN\ffN{\f ^{N}}
\DN\Fmrs{\f _{mrs}}
\DN\Fmrstilde{\tilde{\f }_{mrs}}
\DN\fmrs{f_{mrs}}

\DN\jj{j}
\DN\jmrs{\jj _{mrs}}
\DN\Jmrs{\mathsf{J}_{mrs}}
\DN\kmr{k_{mr}}

\DN\uu{u}
\DN\uN{\uu ^{N}}
\DN\vN{v^{N}}
\DN\vNsinfty{\int_{s\le |x-y|}\vN (x,y)dy }
\DN\vNs{\int_{|x-y|< s}\vN (x,y)dy }
\DN\vvv{v}
\DN\vsinfty{\int_{s \le |x-y|}\vvv (x,y)dy }
\DN\vs{\int_{|x-y|< s}\vvv (x,y)dy }

\DN\supN{\sup_{N\in\N}}
\DN\xyi{|x-y_i|}
\DN\yi{|y_i|}

\DN\gin{\mathrm{g}}
\DN\mug{\mu _{\gin }}
\DN\mugz{\mu _{0}}
\DN\muga{\mu _{\gin ,\mathsf{a}}}
\DN\mugx{\mu _{\gin ,x}}
\DN\mugxx{\mu _{\gin ,\mathbf{x}}}
\DN\mugN{\mu ^{N}_{\gin }}
\DN\mugNx{\mu ^{N}_{\gin ,x}}
\DN\mugNzero{\mu ^{N}_{\gin ,0}}
\DN\mugone{\mu _{\gin }^{[1]}}
\DN\mugNone{\mu ^{N,{[1]}}_{\gin }}
\DN\mugNstar{\mu ^{N*}_{\gin }}
 \DN\muone{\mu ^{[1]}}
 \DN\rg{\rho _{\gin }}
 \DN\rgx{\rho _{\gin ,x}}
 \DN\rgz{\rho _{\gin ,o}}
 \DN\rgn{\rho _{\gin }^n}

 \DN\kg{\mathsf{K}_{\gin }}

\DN\SSS{\mathsf{S}}
\DN\SSSS{\mathbf{S}}
\DN\SSSr{\SSS _{r}}
\DN\SSSrm{\SSS _{r}^{m}}
\DN\SSSrNk{\SSS _{r,N-k}}
 \DN\SSSg{\SSS _{\gin }}
 \DN\SSSSg{\SSSS _{\gin }}

\DN\SN{\SS ^{\mathbb{Z}}}
\DN\SSSksingle{\SSSsi ^{k}}
\DN\SSSsi{\SSS _{\mathrm{s.i.}}}
\DN\SSSone{\SSS ^{\mathbf{1}}}
\DN\SoneSSS{\SS \ts \SSS }

\DN\sS{S}
\DR\SS{\mathbb{R}} 
\DN\Sk{(\SS )^{k}}
\DN\Sr{\sS _{r}}
\DN\Se{\sS _{\epsilon }}
\DN\SR{\sS _{R}}
\DN\SeR{\sS _{\epsilon ,R}}
\DN\Ss{\sS _{s}}
\DN\Sq{\sS _{\q }}
\DN\Srr{\sS _{r}^{k}}
\DN\Sone{\sS _{1}}

\DN\dlog{\mathsf{d}}
\DN\dgin{\dlog ^{\mug }}
\DN\dginN{\dlog ^{\mugN }}
\DN\dginone{\dlog ^{\mug ^1}}
\DN\dginNone{\dlog ^{\mugNone }}

\DN\dmu{\dlog ^{\mu }}
\DN\dmuN{\dlog ^{N}}
\DN\dmuone{\dlog ^{\mu ^1}}

\DN\aaa{\mathsf{a}}
\DN\ssss{(s_i,(s_j)_{j\not=i})}
\DN\sss{\mathsf{s}}
\DN\xsss{(x,\sss )}
\DN\xxxx{(x,(y_j)_{j\in\N })}
\DN\xxxxx{(x,\sum_{j\in\N }\delta _{y_j})}

\DN\hatmux{\hat{\mu }_x}
\DN\hatmuz{\hat{\mu }_0}
\DN\muz{\mu _{0}}
\DN\dxmu{dx \ts \mu }
\DN\muk{\mu ^{{k}}}
\DN\mukg{\mug ^{{k}}}

\DN\RdT{\Rd \ts \SSS }

\DN\E{\mathcal{E}}
\DN\Ea{\E }
\DN\Eak{\E ^{k}}
\DN\Eaone{\E ^{1}}
\DN\Eazero{\E ^{0}}

\DN\DDD{\mathbb{D}}

\DN\iRT{\int_{\Rd \ts \SSS }}
\DN\iR{\int_{\Rd }}
\DN\iT{\int_{\SSS }}

\DN\Rdk{(\Rd)^{k}}
 \DN\PP{\mathsf{P}}

\DN\den{\varsigma}
\DN\gauss{\mathbf{g}}
\DN\self{\alpha }

\DN\Prj{\mathfrak{P}}
\DN\PrjY{\Prj _{\Y }} 
\DN\PrjYr{\Prj _{\Y ,r}} 

\DN\PrjYL{\Prj _{\Y }^{\mathsf{L}}} 
\DN\PrjYLr{\Prj _{\Y ,r}^{\mathrm{L}}} 
\DN\PrjYLrone{\Prj _{\Y ,r}^{\mathrm{L},1}} 
\DN\PrjYLrtwo{\Prj _{\Y ,r}^{\mathrm{L},2}} 
\DN\PrjYLone{\Prj _{\Y }^{\mathrm{L},1}} 
\DN\PrjYLtwo{\Prj _{\Y }^{\mathrm{L},2}}

\DN\yy{Y^{\star }} \DN\yyhat{\hat{X}^{\star }} 

\DN\chieR{\chi _{\epsilon , R}}

\DN\1{dx d\mugz }
\DN\3{ \R\ts\mathsf{S} , \1 }
\DN\8{\frac{\beta }{2}} 
\DN\9{\frac{2}{\beta }}
\DN\6{_{\epsilon , R}}
\DN\7{\overline{\chi }\6 }

\title{Self-diffusion constants of non-colliding  interacting Brownian motions in one spatial dimension} 

\TitleHead{Self-diffusion constants of non-colliding  interacting Brownian motions in $ \mathbb{R}$ }          
\author{\textsc{Hirofumi Osada}
\footnote{Graduate School of Mathematics, Kyushu University
Fukuoka 819-0395, Japan\newline e-mail: 
\texttt{osada@math.kyushu-u.ac.jp}}}
\AuthorHead{Hirofumi Osada}         
\support{Supported in part by a Grant-in-Aid for Scientific Research (KIBAN-A,
No. 24244010) of the Japan Society for the Promotion of Science.}  
\VolumeNo{59}           
\YearNo{2016}           
\PagesNo{253--272}      
\communication{Received January 31, 2016.
Revised May 12, 2016.}      
\begin{document}

\maketitle

\DN\0{ \eqref{:10e}, \eqref{:20c}--\eqref{:20e}, and \eqref{:20z}}

\begin{abstract}      
We prove that self-diffusion constants of interacting Brownian particles in $ \mathbb{R}$  always vanish if the particles do not collide with each other. 
We represent self-diffusion constants by additive functionals of reversible Markov processes as obtained in \cite{o.inv2}. 
\end{abstract}

\section{Introduction}\label{s:1}
We consider infinitely many Brownian particles 
$ \mathbf{X}=(X^i)_{i\in\mathbb{Z}}$ 
moving in $ \Rd $ with interaction $ \Psi $ and inverse temperature 
$ \beta > 0 $. 
Intuitively, $ \mathbf{X}$ is given by the infinite-dimensional stochastic differential equation (ISDE)  
\begin{align}\label{:10a}&\quad 
dX_t^i = dB_t^i -\frac{\beta }{2} \sum_{\jZi } \nabla \Psi (X_t^i-X_t^j) dt ,
\quad (i\in\mathbb{Z})
\end{align}
and  is called an interacting Brownian motion in infinite dimensions \cite{lang.1,lang.2,Sh,fritz,tane,o.isde,o-t.tail}. 
We set the configuration-valued process $ \mathsf{X} $, called  unlabeled dynamics, 
to be 
\begin{align}\label{:10b}&
\mathsf{X}_t = \sum_{i\in\mathbb{Z}} \delta_{X_t^i}
.\end{align}
The dynamics in the present paper are quite general and 
not necessarily given by ISDEs of the form \eqref{:10a}.  
We later present the unlabeled dynamics $ \mathsf{X} $ 
using Dirichlet form theory \eqref{:20b} where 
the labeled dynamics $ \mathbf{X}$ are an 
$ (\Rd )^{\Z }$-valued additive functional of $ \mathsf{X}$ \eqref{:20v}.

We suppose that the system is in equilibrium in the sense that 
\begin{align}&\notag 
\mathsf{X}_t \stackrel[\mathrm{law}]{}{=} \mathsf{X}_0 
\quad \text{ for all }t 
,\end{align}
and that $ \mathsf{X}$ is a $ \mu $-reversible diffusion, 
where $ \mu $ is the distribution of $ \mathsf{X}_0 $. 
Thus $ \mu $ is the equilibrium state of the unlabeled dynamics $ \mathsf{X}$. 
By definition, $ \mu $ is a probability measure on the configuration space $ \mathsf{S}$ over $ \Rd $: 
\begin{align}\label{:10d}&
\mathsf{S} = \{ \mathsf{s}=\sum_i \delta_{s_i}\, ;\, s_i \in \Rd  ,\, 
\mathsf{s}(\{ |s_i |\le r  \} ) < \infty \text{ for all } r \in \mathbb{N}) \} 
.\end{align}
We equip $ \mathsf{S}$ with the vague topology under which $ \mathsf{S}$ 
is a Polish space. 
Throughout the paper we assume that  $ \mu $ is translation invariant:
\begin{align}\label{:10e}&
\mu \circ \theta_x^{-1} = \mu \quad \text{ for each } x \in \Rd 
,\end{align}
where $ \map{\theta _x}{\mathsf{S}}{\mathsf{S}}$ is the shift 
$ \theta _x(\mathsf{s}) = \sum_i \delta_{s_i+x}$ for 
$ \mathsf{s} = \sum_i \delta_{s_i}$. 
By definition, $ \theta_{-x} = \theta_x^{-1}$, and 
$ \theta_x$ is a homeomorphism for each $ x $.

We tag a particle $ X^0 =\{ X_t^0 \} $, say, and study its asymptotic behavior. 
In particular, we investigate the diffusive scaling limit. 
The celebrated Kipnis--Varadhan theory  \cite{KV} 
asserts that tagged particles of reversible systems 
converge to (a constant multiple of)  the Brownian motion $ B $: 
\begin{align}&\notag 
\limz{\ve }\ve X_{\cdot /\ve ^2}^0  = \sigma (\mathsf{s}) B 
,\end{align}
where $ \sigma (\mathsf{s}) $ may depend on the initial configuration $ \mathsf{s}$ of 
 environment seen from the tagged particle, that is, 
$ \mathsf{s} = \sum_{i\not= 0} \delta_{X_0^i-X_0^0}$. 
The average $\alpha [\mu ] $ of $ \sigma ^t\sigma $ 
with respect to the reduced Palm measure $ \mu _0 = \mu (\cdot - \delta_0 | \mathsf{s}(\{ 0 \} )\ge 1 ) $ conditioned at the origin
is called a self-diffusion matrix (see \eqref{:35y}). 

It is known that $ \alpha [\mu ] $ is always positive definite if $ d \ge 2 $ and that 
the interaction $ \Psi $ is of Ruelle's class 
with hard core having positive volume \cite{o.p}. 
It is expected that self-diffusion matrices for point processes with Ruelle's class potentials are always positive definite in multiple dimensions. 
The only known degenerate example in multiple dimensions is Ginibre interacting Brownian motion, in which an infinite-particle system interacts via the two-dimensional Coulomb potential $ \Psi (x) = - \log |x|$ with $ \beta = 2 $ \cite{o.sub}. We remark that the two-dimensional Coulomb potential is not of Ruelle's class because it is unbounded at infinity.

In one dimension, there are degenerate examples such as 
hard rods (Harris \cite{harris}) and Dyson's model (Spohn \cite{sp.1,sp.2}),  
which have scaling orders such that 
$ O(t^{1/4})$ in \cite{harris} and $ O ((\log t)^{1/2})$ in \cite{sp.2}, respectively. 

The main purpose of this paper is to give a sufficient condition for the degeneracy of the self-diffusion constant in one dimension. We prove that $ \alpha [\mu ] = 0 $ if particles do not collide with each other (\tref{l:21}). Let 
$$ \Delta = \{ \mathbf{x}=(x_i)_{i\in\Z } \in \R ^{\Z } \, ;\, 
x_i < x_{i+1} \quad \text{ for all } i \in \Z \} 
.$$
Then the non-collision condition \eqref{:20f} requires that 
$ \mathbf{X}_t \in \Delta $ for all $ 0\le t < \infty $ for all labeled particles 
$ \mathbf{X}=(X^i)_{i\in\Z }$ that start at $ \mathbf{X}_0 \in \Delta $. 
The set $ \Delta $ is very tiny compared with the whole space 
$ \R ^{\Z }$, which is the state space of independent Brownian motions 
$ \mathbf{B}=(B^i)_{i\in\Z }$ and interacting Brownian motions $ \mathbf{X}$ 
with Ruelle's class potentials but without non-collision condition. 
Intuitively, such a small state space of interacting one-dimensional Brownian motions with the non-collision condition results in sub-diffusive behavior. 
Briefly, smallness of the state space implies sub-diffusivity. 
To some extent, this phenomenon resembles the sub-diffusivity of a simple random walk in the incipient infinite cluster, which is a random domain 
enjoying a fractal structure at the critical point of Bernoulli percolation.  

Another purpose of this paper is to give a statement of the main theorem 
in a more natural fashion than previously (\tref{l:37}). 
Traditionally, these problems are  stated for tagged particles $ \{X_t^0\}$ 
as functionals of the stationary environment process 
$ \mathsf{Y} = \{\sum_{i\not= 0} \delta_{X_t^i-X_t^0}\}$, 
with the reduced Palm measure $ \muz $ as the invariant probability measure 
\cite{De,gp,KV,o.p}. 
One statement of the result is that, for any  $ F \in C_b( \CRd )$, 
\begin{align}\label{:11p}& \quad \quad 
\limz{\epsilon} 
\muz (\{\mathsf{s}; 
|\mathsf{E}_{\mathsf{s}}^{\mathsf{Y}}
 [F(\epsilon X_{\cdot /\epsilon ^2}^0 )]  - 
E [F (\sigma (\mathsf{s}) B ) ] | \ge \kappa \} ) 
= 0 
\quad \text{ for each $ \kappa > 0 $}
.\end{align}
Here, the expectation $ \mathsf{E}_{\mathsf{s}}^{\mathsf{Y}} $ is with respect to the distribution $ \mathsf{P}_{\mathsf{s}}^{\mathsf{Y}} $ 
of the environment process 
$ \mathsf{Y} = \{\sum_{i\not= 0} \delta_{X_t^i-X_t^0}\}$ 
starting at $ \mathsf{s}$, and 
$ \muz $ is the initial distribution of the whole environment process 
$ \mathsf{Y} $. 
Furthermore, $ B = \{ B _t\} $ denotes the standard $ d $-dimensional Brownian motion and $ E [\cdot ]$ is the expectation with respect to $ B $. 
This reduction of the problem through the idea of 
\lq\lq tagged particle and environment seen from the tagged particle'' 
is a key idea in Kipnis-Varadhan theory for tagged particle problems. 

We define a label 
$ \lab (\mathsf{s}) = (\lab _i(\mathsf{s}))_{i\in\Z } = (s_i)_{i\in\Z }$ for 
$ \mu $-a.s.\ $ \mathsf{s} \in \mathsf{S}$, and construct the labeled process 
$ \mathbf{X} = (X^i)_{i\in\mathbb{Z}}= \lab_{\mathrm{path}}(\mathsf{X})$ from the unlabeled process $ \mathsf{X}=\sum_{i\in\Z } \delta_{X^i}$ 
(see \eqref{:20v}). 
A refinement of statement \eqref{:11p} is that, for each $ i \in \Z $, 
\begin{align}\label{:11q}& \quad \quad 
\limz{\epsilon} 
\mu (\{\mathsf{s}; 
|\mathsf{E}_{\mathsf{s}} [F(\epsilon X_{\cdot /\epsilon ^2}^i ) ] - 
E [F (\sigma (\mathsf{s}) B )]  | \ge \kappa \} ) 
= 0 
\quad \text{ for each $ \kappa > 0 $}
.\end{align}
Here, $ \mathsf{E}_{\mathsf{s}}$ is the expectation with respect to the distribution $ \mathsf{P}_{\mathsf{s}} $ of the original $ \mu $-reversible diffusion 
$ \mathsf{X}=\sum_{i\in\Z } \delta_{X^i}$ 
given by \eqref{:10b} 
starting at $ \mathsf{s}$. 

We thus state the theorem in terms of the scaling limit of 
each tagged particle $ X^i $ of the original unlabeled dynamics 
$ \mathsf{X} = \sum_i \delta_{X^i}$ (\tref{l:37}). 
To formulate the statement as \eqref{:11q}, we need to prepare a label $ \lab $ 
to choose tagged particles from the unlabeled system.

We remark that the equilibrium state for the unlabeled dynamics is $ \mu $. 
Note that $ \muz $ is not necessary absolutely continuous with respect to $ \mu $. 
Indeed, the Ginibre point process is an example where 
$ \mu $ and $ \muz $ are singular each other \cite{o-s.abs}. 
We nontheless deduce claim \eqref{:11q} from \eqref{:11p}.

\begin{exa}[Ruelle's class potentials]\label{d:11}
Typical examples of the interaction $ \Psi $ in ISDE \eqref{:10a} 
are Ruelle's class potentials, 
where the point processes $ \mu $ are translation invariant canonical Gibbs measures with interaction potential 
$ \Psi $ and inverse temperature $ \beta $. 
If $ d \ge 2$ and $ \Psi $ has a  hard core with positive volume, then 
$ \alpha [\mu ] $ is positive definite for any $ \beta \ge 0 $ \cite{o.p}. 
If $ \Psi $ does not have a hard core, then the positivity of $ \alpha [\mu ] $ 
has only been proved for $ \Psi \in C_0(\Rd )$ and sufficiently small $ \beta $ \cite{De}. 
This result is valid for $ d \ge 1 $. 
It is plausible that the positivity of $ \alpha [\mu ] $ holds for all Ruelle's class potentials 
without any hard core condition or restriction on $ \beta $ if $ d \ge 2 $. 
This problem is still open in this framework. 
\end{exa}

\begin{exa}[Ruelle's class potentials: $ d=1$]\label{d:11a}
Let $ \Psi $ be of Ruelle's class. 
Suppose that $ \Psi \in  C^3 (\Rd \backslash \{ 0 \} )$ 
is non-negative with bounded support 
and satisfies the non-collision condition \eqref{:20e}.  
The point processes $ \mu $ are translation invariant canonical Gibbs measures with interaction potential $ \Psi $ and inverse temperature $ \beta $. 
In \cite{sp.2}, Spohn showed that $ \alpha [\mu ] $ vanishes and 
that the correct scaling is $ \ve X_{t/\ve ^4}^0 $. 
\end{exa}

\begin{exa}[Hard rods in $ \mathbb{R}$]\label{d:12} 
Consider Brownian motions $ \mathbf{B}=(B^i)_{i\in\Z }$ 
in $ \mathbb{R}$, and pose the reflecting boundary condition for the set of 
 multiple points $ \Gamma = \cup_{i\not=j} \{ \mathbf{s}=(s_i); s_i=s_j \} $. 
 Harris proved that $ \alpha [\mu ] = 0 $ and that the correct order is $ \ve X_{t/\ve 4}^0 $ \cite{harris}. 
\end{exa}

\begin{exa}[Dyson's model]\label{d:13}
Another interesting example is Dyson's model in infinite dimensions \cite{sp.1}: 
\begin{align}\label{:1c5}&
dX_t^i = dB_t^i + \limi{R}\frac{\beta }{2} \sum_{j\not=i , |X_t^j|<R}
\frac{1 }{X_t^i-X_t^j }dt 
.\end{align}
The ISDE was solved for $ \beta = 1,2,4$ in \cite{o.isde,o.rm} (weak solution) and in \cite{o-t.tail} (pathwise unique, strong solution). 
The case $ \beta =1,2,4 $ fulfill the assumptions in \tref{l:21}. 
For general $ 1\le \beta < \infty $, 
Tsai \cite{tsai.14} solved \eqref{:1c5} at the level of 
non-equilibrium, pathwise unique strong solutions. 
He did not, however, prove the $ \mu $-reversibility of the associated unlabeled diffusions. 
Hence, these diffusions have not yet been shown to be associated with Dirichlet forms. 
The problem is thus still open for cases other than $ \beta = 1,2,4$. 

In \cite{sp.2}, Spohn proved that 
$ E[|X_t^0-X_0^0|^2] \sim \mathrm{const.} \log t$ as $ t \to \infty $. 
 This suggests that$ X_t \approx (\log t )^{1/2} $ as 
$ t \to \infty $, which is also an open problem. 
\end{exa}

The organization of the paper is as follows. 
In \sref{s:2}, we set up the problem and state the main result (\tref{l:21}). 
In \sref{s:3}, we prepare an invariance principle and state \tref{l:37}.  
In \sref{s:4}, we present a representation of self-diffusion constants. 
In \sref{s:5}, we prove that the self-diffusion constant vanishes under the non-collision condition in one-dimension. 
In \sref{s:6}, we complete the proof of \tref{l:21}.

\section{Set up and the main result}\label{s:2} 

In this section, we set up and state the main theorem (\tref{l:21}). 

Let $ \mathsf{S}$ be the configuration space over $ \Rd $ as in \eqref{:10d}. 
Let $ \mu $ be a point process on $ \Rd $ supported on a set consisting of infinitely many particles. 
We assume that $ \mu $ is translation invariant as in \eqref{:10e}.

A symmetric and locally integrable function 
$ \map{\rho ^n }{(\Rd )^n}{[0,\infty ) } $ is called 
the $ n $-point correlation function of a random point field $ \mu $ 
on $ \Rd $ with respect to the Lebesgue measure if $ \rho ^n $ satisfies 
\begin{align} & \notag 
\int_{A_1^{k_1}\ts \cdots \ts A_m^{k_m}} 
\rho ^n (x_1,\ldots,x_n) dx_1\cdots dx_n 
 = \int _{\Rd } \prod _{i = 1}^{m} 
\frac{\mathsf{s} (A_i) ! }
{(\mathsf{s} (A_i) - k_i )!} d\mu 
 \end{align}
for any sequence of disjoint bounded measurable sets 
$ A_1,\ldots,A_m \in \mathcal{B}(\Rd ) $ and some sequence of natural numbers 
$ k_1,\ldots,k_m $ satisfying $ k_1+\cdots + k_m = n $. 
If $ \mathsf{s} (A_i) - k_i  < 0$, according to our interpretation, 
${\mathsf{s} (A_i) ! }/{(\mathsf{s} (A_i) - k_i )!} = 0$ by convention.

For a function $ f $ on $ \mathsf{S}$, we denote by $ \check{f} $ 
 the symmetric function defined on a subset of 
$ \{\cup_{n=0}^{\infty} (\Rd )^n \}\cup (\Rd )^{\Z } $ such that 
$ \check{f}(s_1,\ldots,) = f (\mathsf{s})$, where  
$ \mathsf{s}=\sum_i \delta_{s_i} \in \mathsf{S}$. 
We say that a function $ f $ on $ \mathsf{S}$ is smooth if 
$  \check{f}$ is smooth, and local if $ f = f\circ \pi_{r}$ for some $ r $, 
where $ \map{\pi_r}{\mathsf{S}}{\mathsf{S}}$ such that 
$ \pi_r (\mathsf{s}) = \mathsf{s} (\cdot \cap \{ |x|\le r \} )$.

Let $ \mathcal{D}_{\circ} $ be the set of all smooth, local functions on 
$ \mathsf{S}$. 
Let $ \mathbb{D} $ be the square field on $ \mathsf{S}$ such that, 
for $ f,g \in \mathcal{D}_{\circ} $, 
\begin{align}\label{:20a}&
\mathbb{D}[f,g] (\mathsf{s})= \frac{1}{2} 
\Big\{ \sum_{i\in \mathbb{Z}} 
\PD{\check{f}}{s_i} \cdot \PD{\check{g}}{s_i} \Big\} (\mathsf{s})
.\end{align}
We use the square field $ \mathbb{D}$ in \eqref{:20a} to define a function $ F_{\mathbb{D}}$ from the space of point processes to the space of bilinear forms on the configuration space $ \mathsf{S}$. Indeed, we set 
\begin{align*}&
F_{\mathbb{D}} (\mu ) = (\Emu , \mathcal{D}_{\circ}^{\mu } ) 
,\end{align*}
where the bilinear form $(\Emu , \mathcal{D}_{\circ}^{\mu } ) $ 
on $ \Lmu $ is such that  
\begin{align}\label{:20b}&
\Emu (f,g)  = \int_{\mathsf{S}} \mathbb{D}[f,g] \mu (d\mathsf{s})
,\\ \notag &
\mathcal{D}_{\circ}^{\mu } = \{ f \in\mathcal{D}_{\circ} \, ;\, 
f \in \Lmu ,\,  \Emu (f,f) < \infty 
\} 
.\end{align}
If $ F_{\mathbb{D}} (\mu ) = (\Emu , \mathcal{D}_{\circ}^{\mu } ) $ is closable on $ \Lmu $, and its closure 
$ (\Emu , \mathcal{D}^{\mu } )$ is a quasi-regular Dirichlet form, then 
by the general theory of Dirichlet forms there exists a $ \mu $-reversible diffusion associated with the Dirichlet space 
 $ (\Emu , \mathcal{D}^{\mu } , \Lmu )$ \cite{mr}.  
Hence we assume that 
\begin{align}\label{:20c}&
\text{$ (\Emu , \mathcal{D}_{\circ}^{\mu } )$ is closable on 
$ \Lmu $}
,\\
\intertext{and }&\label{:20d}
\text{the $ n $-point correlation function $ \rho ^n $ 
of $ \mu $ is locally bounded for each $ n \in \mathbb{N}$}
.\end{align}
We can deduce from \eqref{:20d} that 
the closure $ (\Emu , \mathcal{D}^{\mu } )$ of 
$ (\Emu , \mathcal{D}_{\circ}^{\mu } )$ 
is a quasi-regular Dirichlet form \cite{o.dfa,o.rm}. 
Thus the associated diffusion $ (\mathsf{P},\mathsf{X})$  exists \cite{mr}, 
where $ \mathsf{P}=\{ \mathsf{P}_{\mathsf{s}} \}_{\mathsf{s}\in \mathsf{S}} $ is the family of diffusion measures and 
$ \mathsf{X}=\{ \mathsf{X}_t \}_{t\in[0,\infty)} $ denotes 
the canonical process. 
By construction, $ (\mathsf{P},\mathsf{X})$ is a $ \mu $-reversible diffusion. 

If $ \mu $ is a Poisson point process with Lebesgue intensity, then the associated diffusion is an $ \mathsf{S}$-valued Brownian motion $ \mathsf{B}=\sum_i \delta_{B^i}$. 
In some sense, this correspondence is natural. 
If $ \mu $ is a $ (0, \Psi )$-quasi-Gibbs measure with upper semi-continuous potential $ \Psi $ in the sense of \cite{o.rm}, 
then \eqref{:20c} is satisfied \cite{o.rm,o.rm2}. 
We also remark that a $ (0, \Psi )$-Gibbs measure is 
a $ (0, \Psi )$-quasi-Gibbs measure by definition. 
We present the ISDEs associated with these unlabeled diffusions at the end of this section.

We assume that %
\begin{align}\label{:20e}&
\mathrm{Cap}(\mathcal{N}_1 ) = 0
,\end{align}
where $ \mathcal{N}_1=\{ \mathsf{s}\in \mathsf{S}\, ;\, \mathsf{s}(\{ x \} ) \ge 2 \text{ for some } x \in \Rd \}  $ and 
$ \mathrm{Cap}$ is the 1-capacity of the Dirichlet space 
$ (\Emu , \mathcal{D}^{\mu } ,\Lmu )$. 
Assumption \eqref{:20e} is equivalent to 
\begin{align}\label{:20f}&
\mathsf{P}_{\mathsf{s}}( X_t^i\not=X_t^j \text{ for all }i\not=j ,\ 
t \in [0,\infty) ) = 0 
\quad \text{ for q.e.\! } \mathsf{s}
,\end{align}
where $ \mathsf{X}_t=\sum_i \delta_{X_t^i}$, and 
q.e.\! indicates quasi-everywhere (see \cite{mr,fot.2}). 
 
We refer to Inukai \cite{inu} for the necessary and sufficient conditions for 
\eqref{:20e} in terms of $ \Psi $. This result is for finite-particle systems, but can be used to obtain a precise sufficient condition for \eqref{:20e} 
because the limiting Dirichlet form is a decreasing limit of finite-particle Dirichlet forms \cite{o.dfa}. 
We also refer to \cite{o.col} for the non-collision property of unlabeled diffusions associated with determinantal point processes.

We remark that the translation invariance of $ \mu $ yields the non-explosion of 
each tagged particle $ X^i$ \cite[Theorem 2.5]{o.tp}: 
\begin{align}&\notag 
\mathsf{P}_{\mathsf{s}}( \sup_{t\in[0,T]} |X_t^i| < \infty 
\text{ for all } i \in \mathbb{Z},\ T \in [0,\infty) ) = 0 
\quad \text{ for q.e.\! } \mathsf{s}
.\end{align}

Let $ {\ulab}$ be a function on $ (\Rd )^{\Z }$ such that 
$ \ulab ((s_i)) = \sum_i \delta_{s_i}$ and let 
\begin{align}\label{:20u}&
 \mathsf{S}_{\mathrm{s.i}} = \{ \mathsf{s}\in \mathsf{S} \, ;\, \mathsf{s}(\{ x \} ) \le 1 \text{ for all } x ,\, \mathsf{s}(\Rd )=\infty \} 
.\end{align}
From \eqref{:20e} and the translation invariance of $ \mu $ we see that 
\begin{align}\label{:20p}&
\mu ( \mathsf{S}_{\mathrm{s.i}} ) = 1 
.\end{align}
Let $ \map{\lab }{\mathsf{S}_{\mathrm{s.i}} }{(\Rd )^{\Z }}$ 
be a measurable map such that 
$ \ulab\circ \lab = \mathrm{id}$. We call $ \lab $ a label and write 
$ \lab $ as $ \lab (\mathsf{s}) =(\lab _i(\mathsf{s}))_{i\in\Z }=(s_i)_{i\in\Z }$, 
where $ \mathsf{s} = \sum _{i\in\Z } \delta_{s_i}$.  

\smallskip

\noindent 
\textbf{Example: } 
Let $ \lab =(s_i)_{i\in\Z }$ be a label. The label $ \lab $ is well defined for all 
$ \mathsf{s}=\sum_i\delta_{s_i} \in  \mathsf{S}_{\mathrm{s.i}} $ from the following. 
\\
\thetag{1} 
When $ d=1$, a typical example of the label $ \lab $  is as follows: 
$s_{-1} < s_0 < s_1$ and 
\begin{align*}& 
\quad 
\text{
$ \cdots < s_{-2} < s_{-1} < 0 < s_1 < s_2 < \cdots $
}
.\end{align*}
\thetag{2} 
Another example for $ d \ge 1 $ is: 
\begin{align*}&
 |s_0| < |s_1| < |s_{-1}| < |s_2| < |s_{-2}| < \cdots 
.\end{align*}

\smallskip

We can lift the map $ \lab $ to $ \map{\lab_{\mathrm{path}}}
{C([0,\infty);(\Rd )^{\Z})}{C([0,\infty);(\Rd )^{\Z})}$ in an obvious fashion. 
Indeed, once $ \lab $ is given, the dynamics $ \mathsf{X}_t =\sum_i \delta_{X_t^i}$ can keep the initial label for all $ t \in [0,\infty)$ because the particles neither collide 
with each other nor explode. Hence we write 
\begin{align}\label{:20v}&
\mathbf{X}_t = (X_t^i)_{i\in\mathbb{Z}}= \lab_{\mathrm{path}}(\mathsf{X})_t
.\end{align}
We assume that a label $ \lab =(\lab _i(\mathsf{s}))_{i\in\Z }=(s_i)_{i\in\Z }$ 
is given and fix this throughout the paper. 
If $ \mathsf{X}_0 = \mathsf{s}$ and $ \mathbf{X}$ satisfies \eqref{:20v}, 
then $  \mathbf{X}_0 = (\lab _i(\mathsf{s}) )_{i\in\Z } $ by definition. 
We study the diffusive scaling limit of each tagged particle $ X^i $
of the labeled process 
$ \mathbf{X}=(X^i)_{i\in\Z }$. 

Let $ \mu _x =  \mu (\cdot -\delta_{x}| \mathsf{s}(\{ x\}) \ge 1 ) $ be the reduced Palm measure conditioned at $ x \in \Rd $. 
Let $ \mu ^{[1]} (dxd\mathsf{s})= 
\rho ^1(x)\mu _x (d\mathsf{s} )dx$ be the one-Campbell measure of $ \mu $. 
Note that $ \rho ^1 (x)$ is constant in $ x $ 
 because $ \mu $ is translation invariant by \eqref{:10e}.  
Let $ \nabla _x $ be the nabla in $ x\in \Rd $. 
We regard $ \mathbb{D}$ as a square field on $ C_0^{\infty}(\Rd ) \ot \dz $ 
in an obvious fashion. 
Let 
\begin{align*}&
 \E ^{[1]} (f,g) = \int_{\Rd \ts \mathsf{S}} \Big\{
\frac{1}{2} \nabla_xf \cdot \nabla_x g + \mathbb{D}[f,g] \Big\}
\mu ^{[1]} (dxd\mathsf{s})
,\\&
 \dz ^{[1]} = \{ f \in C_0^{\infty}(\Rd ) \ot \dz ;\, 
f \in \Lmuone ,\, \E ^{[1]} (f,f) < \infty \} 
.\end{align*}
We assume the following: 
\begin{align}\label{:20z}&
\text{$ ( \E ^{[1]} ,  \dz ^{[1]} )$ is closable on $ L^2(\Rd \ts \mathsf{S}, \mu ^{[1]} )$. }
\end{align}

Recall that from \eqref{:20c} and \eqref{:20d} we obtain a $ \mu $-reversible diffusion $ (\mathsf{P},\mathsf{X})$. 
Using the label $ \lab $, we can write $ \mathsf{X}_t = \sum_{i\in\Z } \delta_{X_t^i}$. 
We thus obtain the labeled process $ \mathbf{X} = (X^i)_{i\in\Z }$. 
Our main theorem is the following:
\begin{thm}	\label{l:21}
Assume \0, and assume that $ d = 1$. 
Then, for each $ i \in \mathbb{Z}$, we obtain 
\begin{align}&\notag 
\quad \quad \quad \quad 
\lim_{\epsilon \to 0} \epsilon X_{\cdot /\epsilon ^2}^i = 0 
\quad \text{  weakly in  $\CR $
under $\mathsf{P}_{\mathsf{s}} $ in $ \mu $-probability. }
\end{align}
That is, for any  $ F \in C_b( \CR )$ and  for each $ \kappa > 0 $, 
\begin{align}
&\notag 
\quad \quad 
\limz{\epsilon} 
\mu (\{\mathsf{s}; 
|\mathsf{E}_{\mathsf{s}} [F(\epsilon X_{\cdot /\epsilon ^2}^i ) ]  - 
F (0) | \ge \kappa \} ) 
= 0 
,\end{align}
where $ 0 $ in $ F(0)$ denotes the constant path with value $ 0$. 
\end{thm}

The Dirichlet forms describing 
$( X^0, \sum_{i \not= 0 }\delta_{X^i-X^0}) $ and 
$ \sum_{i \not= 0 }\delta_{X^i-X^0}$ will be given 
in \sref{s:3}. Assumption \eqref{:20z} is necessary for this. 

We emphasize that our framework does not require any ISDE. 
Indeed, only a Dirichlet form constructing the unlabeled diffusion is sufficient. 

If the point process $ \mu $ satisfies the geometric condition below, 
then the unlabeled diffusion given by the Dirichlet form is a solution of the ISDE 
\cite{o.isde}. 
Suppose that $ \mu $ has a logarithmic derivative 
$ \mathsf{d}^{\mu} = \mathsf{d}^{\mu}(s,\mathsf{s}) $ 
in the sense of \cite{o.isde}. Then the labeled dynamics $ \mathbf{X}$ 
are described by the infinite-dimensional stochastic differential equation 
\begin{align}
&\notag 
dX_t^i = dB_t^i + \frac{1}{2} 
\mathsf{d}^{\mu } (X_t^i,\mathsf{X}_t^{\diamondsuit ,i})dt 
,\end{align}
where, for $ \mathbf{X}=(X_t^i)_i$, we set 
\begin{align}\notag 
&
\mathsf{X}_t^{\diamondsuit ,i} = \sum_{\jZi  } \delta_{X_t^j}
.\end{align}
If $ \mu $ is a canonical Gibbs measure with inverse temperature $ \beta $ 
and potential $ \Psi $, then 
\begin{align*}& \quad \quad \quad \quad 
\mathsf{d}^{\mu } (x,\mathsf{s}) = - \beta \sum_i \nabla _x \Psi (x,s_i)
\quad \quad (\mathsf{s}=\sum_i \delta_{s_i})
.\end{align*}
If $ \mu $ is a Ginibre point process or a Sine$ _\beta $ point process, then 
$ \mathsf{d}^{\mu } $ is given by 
\begin{align*}&
\mathsf{d}^{\mu } (x,\mathsf{s}) = \beta \limi{R}
\sum_{|s_i|<R} \frac{x-s_i}{|x-s_i|^2}
.\end{align*}
This justifies the intuition  such that the interaction potentials of these point processes are logarithmic function $ \Psi (x,y) = -\log |x-y| $. 

We remark that our framework \cite{o.inv2} 
is very general, and contains many examples beyond Gibbs measures and 
point processes with pairwise interactions. 
For example, if $ \mu $ is a distribution of the zero points of planar Gaussian analytic functions (GAF), then its logarithmic derivative would not be given by a two-body potential 
$ \Psi $. 
We can still apply our result to this model. 
We plan to study this problem in a forth coming paper. 
We refer to \cite{GP} for discussion of GAFs.

\section{Invariance principle and self-diffusion matrix}
\label{s:3}

In this section, we quote a general theorem on an invariance principle for additive functionals of reversible Markov processes from \cite{o.inv2}, and present 
a refinement corresponding to \eqref{:11q}. 

Throughout this section, we assume \0. 
That is, we make the assumptions in \tref{l:21} except $ d = 1 $. 
We suppress this in the statements of the lemmas in this section. 

Let $ (\mathsf{P},\mathsf{X})$ be the diffusion given by the Dirichlet form 
$ (\Emu , \mathcal{D}^{\mu } )$ on $ \Lmu $ as in \sref{s:2}. 
Let $ \lab $ be a label, and write $ \mathsf{X}=\sum_{i\in\Z } \delta _{X^i}$. 
Let $ \mathbf{X} = (X^i)_{i\in\Z }$ be the associated labeled dynamics. 

A standard device for the tagged particle problem for interacting Brownian motions 
is to introduce processes of the environment seen from the tagged particles 
\cite{De, gp,KV,o.inv1,o.inv2,o.p}. 
Following this, we define a change of coordinates for $ \mathbf{X}$ as follows. 
Let $ \mathbb{Z}^*=\Z \backslash \{ 0 \}$, and set 
\begin{align}\label{:30y}& 
X = X^{0} ,\quad Y^{i} = X^{i}-X^{0} ,\  (\iZs )
.\end{align}
Then $ X $ denotes the tagged particle and 
$ \mathbf{Y}=(Y^{i})_{\iZs }$ 
is the (labeled) environment seen from the tagged particle. 
Let $ \mathsf{Y}=\{ \mathsf{Y}_{t} \} $ 
be the unlabeled process associated with $ \mathbf{Y}=(Y^{i})_{\iZs }$:  
\begin{align}\label{:30z}&
 \mathsf{Y}_{t}= \sum_{\iZs } \delta_{Y^{i}_{t}}
= \sum_{\iZs } \delta_{X_{t}^{i}-X_{t}^{0}}
.\end{align}
Then $ \mathsf{Y}$ is the process 
representing the (unlabeled) environment seen from the tagged particle $ X=X^{0}$. 
We call $ \mathsf{Y}$ the environment process. 
We also call the pair $ (X , \mathsf{Y})$ 
the tagged particle and environment process.

\begin{rem}\label{r:30}
If $ \mathbf{X}$ is given by \eqref{:10a}, 
then from \eqref{:30y} we see that $ (X , \mathbf{Y} )$ 
is given by 
\begin{align}\notag 
dX_{t} &= dB_t^0 - \8 
 \sum_{j \in \Z ^*  } \nabla \Psi  (Y^{j}_{t})
dt 
,\\\notag  
dY^{i}_{t} &= 
  \sqrt{2} \, d\tilde{B}^i  
- \8  \nabla \Psi  (Y^{i}_{t}) dt 
 + \8 
 \sum_{j \in \Z ^*  } 
\nabla \Psi  (Y^{j}_{t}) dt 
 - \8 
 \sum_{j \in \Z ^*  } 
\nabla \Psi  (Y^{i}_{t} - Y^{j}_{t}) dt 
,\end{align}
where $ \{ \tilde{B}^i  \}_{\iZs } $ 
are $ d $-dimensional Brownian motions given by  
\begin{align}&\notag %
\tilde{B}_{t}^i = \frac{1}{\sqrt{2}} (B^{i}_{t} - B^{0}_{t}) 
.\end{align}
We remark that $ \{ \tilde{B}^i  \}_{\iZs } $ 
are not independent but only identically distributed random variables 
equivalent to standard Brownian motion. 
\end{rem}

Using \eqref{:30y} and \eqref{:30z}, we have constructed dynamics 
$ \mathbf{Y}$, $ (X,\mathbf{Y})$, $ \mathsf{Y}$, and $ (X , \mathsf{Y})$ from 
$ \mathsf{X}$. 
We now specify the Dirichlet forms associated with 
$ \mathsf{Y}$ and $ (X , \mathsf{Y})$. 
 
We remark that, although $ \mathbf{Y}$ and 
$ (X,\mathbf{Y})$ are also diffusions with state space 
$ (\Rd )^{\Z } $ and $\Rd \ts (\Rd )^{\Z } $, respectively, 
there exist no associated Dirichlet spaces 
because of the lack of suitable invariant measures. For example, 
if $ \mu $ is a Gibbs measure with interaction $ \Psi $, then 
such measures $ \widetilde{\mu}_0$ and $ dx \ts \widetilde{\mu}_0$ 
for $ \mathbf{Y}$ and $ (X , \mathbf{Y})$ are loosely given by 
\begin{align*}&
\widetilde{\mu}_0  = \frac{1}{\mathcal{Z} } 
\exp \{-\beta \Big( \sum_{i < j,\, i,j\in \Z } \Psi (y_i-y_j)-
\sum_{k\in\Z } \Psi (0-y_k) \Big)  \} \prod_{l\in\Z }dy_l
.\end{align*}
This cannot be justified because of the presence of 
the infinite product of Lebesgue measures $ \prod_{l\in\Z }dy_l$. 
In contrast, $ \mathsf{Y}$ and $ (X , \mathsf{Y})$ are diffusions 
with invariant measures $ \muz $ and $ dx \ts \muz $, respectively. 
As a result, they have associated Dirichlet spaces. 
This fact is key to analysis 
in the Dirichlet form version of Kipnis--Varadhan theory \cite{o.inv2}. 

Once the Dirichlet forms describing the processes 
$ \mathsf{Y}$ and  $ (X , \mathsf{Y})$ have been established, 
we can dispense with ISDE \eqref{:10a}, which yields the generality of our result. 
In fact, the process $ \mathbf{X}$ in \tref{l:21} 
is not necessary given by ISDE \eqref{:10a}. 
Thus our framework is much more general 
than the classical one in \cite{De} and \cite{gp}. 

\smallskip

Let $ \Dsft =(\Dsft _1,\ldots,\Dsft _d)$, where 
$ \map{\Dsft _k}{\dz }{\dz }$ is such that 
\begin{align}
&\notag 
\Dsft _kf (\mathsf{s}) = \limz{\epsilon} \frac{1}{\epsilon} 
\{ f (\theta_{\epsilon \mathbf{e}_k} (\mathsf{s} )) - f (\mathsf{s})\} 
,\end{align}
and $ \mathbf{e}_k$ is the $ k$th unit vector in $ \Rd $. We set 
\begin{align}
&\notag 
\Dsft [f,g] =   \half \Dsft f \cdot \Dsft g 
.\end{align}
Let $ \DDD $ be defined as in \eqref{:20a}. 
Let $\DY $ be the square field on $\dz $ such that  
\begin{align}
\label{:31y}&   
\DY [f,g] =  \Dsft [f,g] +  \, \DDD [f,g]
.\end{align}
Let $ (\EY ,\dYz )$ be the bilinear form such that   
\begin{align}
&\notag 
\EY (f,g) = \int_{\SSS }  \DY [f,g] d\mugz 
,\\ \notag &
\dYz = \{ f \in \Lmuz \cap \dz \, ;\,  \EY (f,f) < \infty \} 
.\end{align}

The next lemma is a special case of a result 
for translation invariant diffusions on $ \SSS $ in \cite{o.tp}, and 
gives a Dirichlet form for $ \Y $. 
\begin{lem}[{\cite[Th.\! 2.6 \thetag{1}, Th.\! 2.7 \thetag{2.33}]{o.tp}}] 
\label{l:31} \quad {} \\
\thetag{1} $ (\EY ,\dYz )$ is closable on $ L^2(\mathsf{S}, \mugz )$. 
\\
\thetag{2} $ \mathsf{Y} $ in \eqref{:30z} is a diffusion associated with 
$ (\EY , \dY  )$ on $ L^2(\mathsf{S}, \mugz )$, 
where $ (\EY , \dY  )$ is the closure of 
$ (\EY ,\dYz )$ on $ L^2(\mathsf{S}, \mugz )$. 
\end{lem}

We next specify the Dirichlet space associated with 
the coupled process $ ( X , \mathsf{Y})$. 
We naturally regard $ \nabla _x $ and $ \Dsft $ as operators on 
$ C_0^{\infty}(\Rd  ) \ot \dz $. 
For $  f , g \in  C_0^{\infty}(\Rd  ) \ot \dz $, we set 
\begin{align}
&\notag 
 (\nabla_x - \Dsft ) [ f , g ] = \frac{1}{2} (\nabla_x - \Dsft ) f \cdot (\nabla_x - \Dsft ) g 
,\\
&\notag 
\DXY [ f , g ] =  (\nabla_x - \Dsft ) [ f , g ]  +  \, \DDD [ f , g ]
,\\   
&\notag 
\EXY ( f , g ) = \int_{\Rd  \ts \SSS } \DXY [ f , g ]  \1  
.\end{align}
Applying a general result in \cite{o.tp} to translation invariant diffusions on $ \SSS $, we obtain: 
\begin{lem}[\cite{o.tp}, Th.\! 2.6] \label{l:32}
$ (\EXY , C_{0}(\Rd )\ot \dz )$ is closable on $ L^2( \3  )$. 
\end{lem}

We denote by $ \PY _{\mathsf{s}} $ the distribution of the diffusion 
$ \mathsf{Y} = \{\Y _t \} $ starting at $ \mathsf{s}$ given by the Dirichlet form $ (\EY , \dY  )$ on $ L^2(\mathsf{S}, \mugz )$. 
We denote by $ \PXY  _{(x,\mathsf{s}) } $ the distribution of the diffusion 
$  (X,\Y ) = \{  (X_t,\Y _t) \} $ starting at $ (x,\mathsf{s}) $ given by the Dirichlet form $ (\EXY , \dXY  )$ on $ L^2( \3  )$.  
By the general theory of Dirichlet forms \cite{fot.2}, 
$ \PY _{\mathsf{s}} $ and $ \PXY  _{(x,\mathsf{s}) } $ are unique 
up to quasi-everywhere starting points. 
The next two lemmas show the existence of suitable versions of these diffusion measures. 

The next lemma explains the relationship between 
$ \mathsf{Y}$ and $ (X,\mathsf{Y})$ and recalls the identities involving 
$  \PY _{\mathsf{s} } $ and $ \PXY  _{(x,\mathsf{s}) } $. 
We set 
\begin{align*}&
X - X_0 = \{ X _t - X _0 \}_{t\in[0,\infty)} 
.\end{align*}
\begin{lem}[{\cite[Lem.\! 2.3]{o.inv2}}] \label{l:33} 
The diffusions $ \PY _{\mathsf{s} } $ and $ \PXY _{(x,\mathsf{s}) } $ 
satisfy the following: 
\begin{align}
&\notag 
 \PY _{\mathsf{s} } =  \PXY _{(x,\mathsf{s}) } ( \Y \in \cdot ) 
\quad \text{ for each $ x \in \Rd  $}
 ,\\
&\notag 
 \PXY _{(0 ,\mathsf{s}) }  ((X - X _0 ,\Y ) \in \cdot )= 
 \PXY _{(x,\mathsf{s}) } ((X - X _0 ,\Y ) \in \cdot ) 
\quad \text{ for each $ x \in \Rd  $}
.\end{align}
\end{lem}

We next clarify the relationship between the original diffusion $ \mathsf{X}$ 
and the diffusion  $ (X,\mathsf{Y})$. 
Let $ \mathsf{S}_{\mathrm{s.i}} $ be defined as in \eqref{:20u}, and let 
\begin{align} 
&\notag 
\mathsf{S}_x = \{ \mathsf{s} \in  \mathsf{S}_{\mathrm{s.i}} ; 
\mathsf{s}(\{ x \} ) = 1 \} 
.\end{align}
For $ \mathsf{s}\in \mathsf{S}_x$ 
 and a label $ \lab (\mathsf{s}) = (s_i)_{i\in\Z }$, 
we set $ i(\mathsf{s},x) \in \Z $ such that 
\begin{align}\label{:34z}&
\text{$ s_{i(\mathsf{s},x)} = x$.  }
\end{align}
Let 
$ \mathsf{P} = \{ \mathsf{P}_{\mathsf{s}}\}_{\mathsf{s}\in \mathsf{S}} $ 
be the distribution of the original unlabeled diffusion $ \mathsf{X}$ 
given by the Dirichlet form 
$ (\Emu , \mathcal{D}^{\mu } )$ on $ \Lmu $ as before. 
Because $ \mu $ is translation invariant, 
there is a version of $ \mathsf{P}_{\mathsf{s}}$ such that 
\begin{align*}& \quad \quad \quad \quad 
\mathsf{P}_{\mathsf{s}} \circ \theta_x^{-1} = 
\mathsf{P}_{\theta_x(\mathsf{s})} \quad \text{ for all } x\in \Rd \text{ and } 
\mathsf{s} \in \mathsf{S}
.\end{align*}
\begin{lem}[{\cite[Th.\! 2.7 \thetag{2.32}]{o.tp}}] \label{l:34} 
Let $ i(\mathsf{s},x)$ be defined as above. Then, for each $ x \in \Rd $ and 
$ \mathsf{s}\in \mathsf{S}_x $,   
(a version of) $ \PXY _{(x,\mathsf{s})}$ satisfies 
\begin{align}\label{:34a}&
\mathsf{P}_{\mathsf{s}} (X^{i(x,\mathsf{s})} \in \cdot ) = 
\PXY _{(x,\theta_{-x} (\mathsf{s}-\delta_x)) } (X \in \cdot )
.\end{align}
\end{lem}

We use an invariance principle obtained in \cite{o.inv2}. 
Applying \cite[Th.\! 1, Lem.\!  5.5]{o.inv2} 
to $ \PY $ and $ \PXY $, we obtain: 
\begin{lem}\label{l:35}
There exists a non-negative definite matrix-valued function $ \hat{a}$ 
such that, for each $ x $, 
\begin{align}\label{:35a}& \quad \quad \quad 
\limz{\ve } (\ve \X _{\cdot /\ve ^2 }  -  \ve \X _{0/\ve ^2 })
=  \sqrt{\hat{a}(\sss )} B  
\quad \text{ in law in $ C([0,\infty);\Rd )$}
\end{align}
under $  \PXY _{(x ,\mathsf{s}) } $ in 
$ \mugz $-probability. 
That is, for any $ F \in C_b(  C([0,\infty);\Rd ) )$ and each $ \kappa > 0 $, 
\begin{align}\label{:35b}& \quad \quad 
\limz{\epsilon} 
\mugz (\{\mathsf{s}; 
|\mathsf{E}^{X\mathsf{Y}} _{(x ,\mathsf{s})} 
[F(\ve \X _{\cdot /\ve ^2 }  -  \ve \X _{0/\ve ^2 }) ]  -
E [F ( \sqrt{\hat{a}(\sss )} B) ] 
| \ge \kappa \} ) 
= 0 
.\end{align} 
\end{lem}

From \lref{l:35}, 
we introduce the self-diffusion matrix $ \selfmu $ given by 
\begin{align}\label{:35y}&
 \selfmu  = \int_{\SSS } \hat{a} (\sss )\mugz (d\sss )
.\end{align}
\begin{lem} \label{l:36}
Let $ x \in \Rd $ and set $ \hatmux  = \mu (\cdot | \mathsf{s}(\{ x \} )=1 )$. 
Then 
\begin{align}
&\notag 
\quad \quad \quad 
\limz{\ve } 
 \ve X_{\cdot /\ve ^2}^{i(x,\mathsf{s})} =  \sqrt{\hat{a}(\sx )} B 
\quad \text{ in law in $ C([0,\infty);\Rd )$}
\end{align}
under $  \PP _{\mathsf{s}} $ in $ \hatmux  $-probability. 
That is, for any $ F \in C_b(  C([0,\infty);\Rd ) ) $ and each $ \kappa > 0 $, 
\begin{align}\label{:36b}& \quad \quad 
\limz{\epsilon} \hatmux   (\{\mathsf{s}; |\mathsf{E}_{\mathsf{s}} 
[F(\ve \X _{\cdot /\ve ^2 }^{i(x,\mathsf{s})}) ]  - E [F ( \sqrt{\hat{a}(\sx )} B ) ] 
| \ge \kappa \} ) 
= 0 
.\end{align} 
\end{lem}
\begin{proof} 
Note that 
$ \ve X_{0/\ve ^2}^{i(x,\mathsf{s})}=x $ for $ \mathsf{P}_{\mathsf{s}}$-a.s.\ and 
$ \ve X_{0/\ve ^2}=x $ for 
$ \PXY _{(x,\theta_{-x} (\mathsf{s}-\delta_x)) }$-a.s.. 
From this and \eqref{:34a} in \lref{l:34}, we see that 
\begin{align}\label{:36d}
\mathsf{P}_{\mathsf{s}} \circ 
( \ve X_{\cdot /\ve ^2}^{i(x,\mathsf{s})} - 
\ve X_{0/\ve ^2}^{i(x,\mathsf{s})} )^{-1}
 &= 
\PXY _{(x,\theta_{-x} (\mathsf{s}-\delta_x)) } 
\circ ( \ve X_{\cdot /\ve ^2} - \ve X_{0/\ve ^2})^{-1}
.\end{align}
We easily see that 
\begin{align}\label{:36e}&
\hatmux   \circ \{ \theta_{-x} (\mathsf{s}-\delta_x) \} ^{-1} 
= \muz 
.\end{align}
Using \eqref{:36d} and \eqref{:36e} and applying \eqref{:35b} in \lref{l:35}, 
we obtain, for each $ x $, 
\begin{align}\label{:36f}&
\limz{\ve }
\hatmux   (\{\mathsf{s}; \left|\mathsf{E}_{\mathsf{s}} 
[F(\ve \X _{\cdot /\ve ^2 }^{i(x,\mathsf{s})}) ]  - E [F ( \sqrt{\hat{a}(\sx )} B ) ] 
\right| \ge \kappa \} ) 
\\ \notag =& \limz{\ve }
\hatmux   (\{\mathsf{s}; \left|\mathsf{E}_{\mathsf{s}} 
[F(\ve \X _{\cdot /\ve ^2 }^{i(x,\mathsf{s})} -
    \ve \X _{0 /\ve ^2 }^{i(x,\mathsf{s})} 
   ) ] - E [F ( \sqrt{\hat{a}(\sx )} B ) ] 
\right| \ge \kappa \} ) 
\\ \notag = & 
\limz{\epsilon} 
\hatmux (\{\mathsf{s}; 
\left|\mathsf{E}^{X\mathsf{Y}} _{(x,\theta_{-x} (\mathsf{s}-\delta_x)) } 
[ F(\ve \X _{\cdot /\ve ^2 }  -  \ve \X _{0/\ve ^2 })]  -
E [F ( \sqrt{\hat{a}(\sx )} B) ] 
\right| \ge \kappa \} ) 
\\ \notag = & 
\limz{\epsilon} 
\mugz (\{\mathsf{s}; 
\left|\mathsf{E}^{X\mathsf{Y}} _{(x ,\mathsf{s})} 
[ F(\ve \X _{\cdot /\ve ^2 }  -  \ve \X _{0/\ve ^2 })]  -
E [F ( \sqrt{\hat{a}(\sss )} B) ] 
\right| \ge \kappa \} ) 
\\
= & \notag \, 0 
.\end{align}
We  immediately deduce \eqref{:36b} from \eqref{:36f}. 
\end{proof}

\begin{thm}\label{l:37}
Assume \0. 
Then, for each $ i \in \mathbb{Z}$, 
\begin{align}\label{:37a}& \quad \quad \quad \quad 
\lim_{\epsilon \to 0} \epsilon X_{\cdot /\epsilon ^2}^i = \sqrt{\hat{a}(\si )} B 
\quad \text{  weakly in } C([0,\infty);\Rd ) 
\end{align}
under $\mathsf{P}_{\mathsf{s}} $ in $ \mu $-probability. That is, 
for any $ F \in C_b(  C([0,\infty);\Rd ) ) $ and $ \kappa > 0 $ 
\begin{align}\label{:37b}& \quad \quad 
\limz{\epsilon} 
\mu (\{\mathsf{s}; 
|\mathsf{E}_{\mathsf{s}} (F(\epsilon X_{\cdot /\epsilon ^2}^i )) -
F ( \sqrt{\hat{a}(\si )} B )
| \ge \kappa \} ) 
= 0 
.\end{align} 
\end{thm}

\begin{proof}
Let $ i (\mathsf{s},x)$ be as in \eqref{:34z}. 
Set $ \hatmux  = \mu (\cdot | \mathsf{s}(\{ x \} )=1 )$ as in \lref{l:36}. 
Recall that $ \mu ( \mathsf{S}_{\mathrm{s.i}} ) = 1 $ by \eqref{:20p}.  
Let $ \lab (\mathsf{s}) = (\labi (\mathsf{s}) )_{i\in\Z }$ be the label as before. 

Without loss of generality, we can and do assume $ i=0 $ in \eqref{:37a}. 
By a straightforward calculation, we have 
\begin{align}\label{:37c}
\mu (X^0 \in \cdot ) = &
\int_{\Rd } \mu (X^0 \in \cdot \, | \labz (\mathsf{s}) = x ) \, \mu \circ \labz ^{-1}(dx)
\\\notag 
= &
\int_{\Rd } \mu (X^0 \in \cdot \, | \labz (\mathsf{s}) = x ,\, 
\mathsf{s} (\{ x \} ) = 1 ) 
\, \mu \circ \labz ^{-1}(dx) 
\\\notag 
= &
\int_{\Rd } \hatmux ( \{ X^{i(\mathsf{s},x)} \in \cdot \} \cap\{ 
i(\mathsf{s},x) = 0 
 \} \, 
| \labz (\mathsf{s}) = x ) \, \mu \circ \labz ^{-1}(dx)
\\\notag 
 = &
 \int_{\Rd } \hatmux (X^{i(\mathsf{s},x)} \in \cdot \, | \labz (\mathsf{s}) = x ) \, \mu \circ \labz ^{-1}(dx)
.\end{align}

Let $ F \in C_b (\CRd ) $, and set 
\begin{align}\label{:37e}& 
G_{\ve }^1(\mathsf{s}) = 
\mathsf{E}_{\mathsf{s}} [F ( \ve X_{\cdot/ \ve^2}^0)] - 
E [F( \sqrt{\hat{a}(\si )} B ) ]
,\\ \notag &
G_{\ve }^2(\mathsf{s}) = 
\mathsf{E}_{\mathsf{s}}  [F ( \ve X_{\cdot/ \ve^2}^{i(\mathsf{s},x)})] - 
E [F( \sqrt{\hat{a}(\sx )} B ) ]
.\end{align}
For each $ \kappa > 0 $, we see from \eqref{:37c} that 
\begin{align}\label{:37n}&
\mu (\{\mathsf{s};  |  G_{\ve }^1(\mathsf{s}) | \ge \kappa \}) 
= 
\int_{\Rd } \hatmux ( \{\mathsf{s}; 
|G_{\ve }^2(\mathsf{s}) | \ge \kappa \} 
  \, 
| \labz (\mathsf{s}) = x ) \, \mu \circ \labz ^{-1}(dx)
.\end{align}
Let $ S_R =\{ x; |x| \le R  \} $. 
For any $ \upsilon > 0 $, take $ R = R (v)$ such that
\begin{align}& \notag 
\mu \circ \labz ^{-1} (S_R^c) \le \upsilon
. \end{align}
Then 
\begin{align}\label{:37f}&
\int_{S_R^c} 
\hatmux  (\{\mathsf{s};  |  G_{\ve }^2(\mathsf{s}) | \ge \kappa \} 
| \labz (\mathsf{s}) = x ) \, \mu \circ \labz ^{-1}(dx)  \le \upsilon 
.\end{align}
It is not difficult to see that, for each $ i , R \in \N $, 
\begin{align} &  \notag 
\quad \quad 
\int_{ S_R}\frac{1}{\hatmux (\lab _i  (\mathsf{s}) = x )} 
\mu \circ \lab _i ^{-1}(dx) 
= \int_{ S_R} \rho ^1 dx 
 < \infty 
.\end{align}
Using this and \lref{l:36}, we apply the bounded convergence theorem 
to obtain 
\begin{align}\label{:37g} & 
\limsupz{\ve } 
\int_{ S_R } \hatmux ( \{\mathsf{s}; 
|G_{\ve }^2(\mathsf{s}) | \ge \kappa \} 
  \, 
| \labz (\mathsf{s}) = x ) \, \mu \circ \labz ^{-1}(dx)
\\ \notag 
\le &
\limsupz{\ve } 
\int_{ S_R} 
\hatmux  (\{\mathsf{s}; 
|G_{\ve }^2(\mathsf{s}) | \ge \kappa \} ) 
\frac{1}{\hatmux (\labz (\mathsf{s}) = x )}
\, \mu \circ \labz ^{-1}(dx)  
\\ \notag = & 
\int_{ S_R} \limsupz{\ve }
\hatmux  (\{\mathsf{s}; 
|G_{\ve }^2(\mathsf{s}) | \ge \kappa \} ) 
\frac{1}{\hatmux (\labz (\mathsf{s}) = x )} 
\, 
\mu \circ \labz ^{-1}(dx)  
\\ \notag 
= & \, 0
.\end{align}

Putting \eqref{:37e}--\eqref{:37g} together we have 
\begin{align}\label{:37h}& 
\limsupz{\ve }\mu (\{\mathsf{s};  |  G_{\ve }^1(\mathsf{s}) | \ge \kappa \}) 
\le \upsilon 
.\end{align}
Because $ \upsilon $ is arbitrary, we see that the left-hand side of \eqref{:37h} 
equals zero. 
Together with \eqref{:37e}, this yields \eqref{:37b}. 
We have thus completed the proof. 
\end{proof}

\begin{rem}\label{r:41} \thetag{1} 
When using Kipnis--Varadhan theory, one has to assume the existence of 
the mean forward velocity $ \varphi \in L^2(\mu _0)$ of tagged particles and that 
\begin{align*}&
|\int_{\mathsf{S}} \varphi f d\mu _0 | \le 
C \EY (f,f) ^{1/2} 
\quad \text{ for all } f \in \dY 
\end{align*}
for some constant $ C $. 
Roughly speaking, the existence of  the mean forward velocity is 
captured by the condition 
\begin{align*}&
\limz{t} \frac{1}{t}E_{\mathsf{s}} [X_t^0-X_0^0] 
=: \varphi (\mathsf{s}) 
\quad \text{ in } \Lmuz 
,\end{align*}
the average of which must vanish with respect to 
the reduced Palm measure $ \muz $. 
It is often difficult to check these conditions for interacting Brownian motions with singular or long range potentials. 
Indeed, one essentially uses the fact that the dynamics are given by a strong solution of an ISDE \eqref{:10a}.  

In \cite{o.inv2}, we developed a Dirichlet form version of Kipnis--Varadhan theory 
that allows us to verify these conditions and generalize the results themselves. 
In fact, only the existence of the coupled Dirichlet form 
$ (\EXY , \dXY  )$ on $ L^2( \3  )$ was necessary 
as a substitute for the mean forward velocity condition above.  
This follows from \0 as we see in \lref{l:32}. 
\\
\thetag{2} In \cite{o.inv2}, Dirichlet forms are assumed to satisfy the strong sector condition, which is a generalization of reversibility. 
It is easy to see that \tref{l:37} holds under the strong sector condition. 
\end{rem}

\section{Representation of the self-diffusion constant $ \selfmu $}  
\label{s:4}

Before quoting a representation theorem for the self-diffusion matrix $ \selfmu $ from \cite{o.inv2}, 
we introduce the quotient Dirichlet form associated with $ (\EY , \dY ) $. 
Although the ideas in this section are valid for $ d \ge 1$, 
we restrict our discussion to $ d=1$ for simplicity. 
We refer the reader to $ \cite{o.inv2}$ for the general case.

Recall the decomposition $ \DY  =  \Dsft  +  \, \DDD $ in \eqref{:31y}, 
and choose bilinear forms on $ \dYz $ such that 
\begin{align}
&\notag 
\EYone (f,g) = \int_{\SSS } \Dsft [f,g] d\mugz ,\quad 
\EYtwo (f,g) = \int_{\SSS }  \DDD [f,g] d\mugz 
.\end{align}
We can naturally extend 
the domain $ \dYz $ of the operator $ \map{\Dsft }{\dYz }{\Lmuz }$ to $ \dY $, 
where these bilinear forms are represented by the Dirichlet form $ \EY $. 
\begin{align}\label{:40b}&
\EY = \EYone + \EYtwo 
.\end{align}
Let $ \dYttwo = \dY / \EYtwo $ be the quotient space of $ \dY $
with the equivalence relation 
$ \sim_{\EYtwo }$ given by $ \EYtwo $, 
that is, $ f \sim_{\EYtwo } g $ if and only if $ \EYtwo (f-g,f-g) = 0 $. 
Let $ \dtYz $ be a vector space 
\begin{align}
&\notag 
\dtYz = \{ \mathsf{f}=(\Dsft f , f/\sim _{\EYtwo } ) 
\in \Lmuz \ts \dYttwo \, ;\, 
f \in  \dY  \} 
\end{align}
with inner product
\begin{align*}&
\EYt (\ff , \gg ) = (\Dsft f ,\Dsft g )_{\Lmuz } + \EYtwo ( f , g )
.\end{align*}
Let $ \dtY $ be the completion of $ \dtYz $ with inner product $ \EYt $ as above. 
Note that $ \EYt (\ff , \gg ) = \EY ( f , g )$. Hence 
$ (\dtY ,\EYt )$ gives a representation of the quotient Hilbert space of $ \dY $ with inner product $ \EY $. 

Let $ (\EYtone ,\dtY )$ and $ (\EYttwo ,\dtY )$ be the quotient bilinear forms of $ (\EYone ,\dY )$ and $ (\EYtwo ,\dY )$ 
defined in the same manner as $ (\EYt , \dtY )$. 
By definition, the domain of these bilinear forms is $ \dtY $, and 
\eqref{:40b} yields the representation 
\begin{align}
&\notag 
\EYt = \EYtone + \EYttwo 
.\end{align}
For $ \ff = (\ff _1,\ff _2) $ and $\gg=(\gg _1,\gg _2)\in \dtY $, we have 
$ \EYt (\ff , \gg ) = \EYtone (\ff _1 , \gg _1 ) + \EYttwo (\ff _2, \gg _2) $. 

For $ (f_1,f_2), (g_1,g_2) \in \Lmuz \ts \dYttwo $ we extend the domain of $ \EYt $ so that 
\begin{align}& \notag 
\EYt ((f_1,f_2) , (g_1,g_2)) = 
(f_1 ,g_1 )_{\Lmuz } + \EYttwo ( f_2 , g_2 )
.\end{align}
We quote the following lemma from \cite{o.inv2}: 
\begin{lem} \label{l:41} 
There exists a unique solution $ \chi \in \dtY $ of the equation 
\begin{align}\label{:41e}&
\EYt (\chi , \gg ) = \EYt ((\ee ,0), \gg ) 
\quad \text{ for all } \gg \in \dtY 
.\end{align}
\end{lem}

\begin{proof}
Because $ F (\gg ) = \EYt ((\ee ,0), \gg ) $ 
can be regarded as a bounded linear functional 
of the Hilbert space $ \dtY $ with inner product $ \EYt $, 
\lref{l:41} is obvious from the Riesz theorem. 
\end{proof}

\begin{lem} \label{l:42} 
Let $ \chi =(\chi _1,\chi _2) \in \dtY $ 
be the unique solution of equation \eqref{:41e}. 
Then the self-diffusion constant $ \selfmu $ is given by 
\begin{align}\label{:42b}&
 \half  \selfmu  = 
\EYtone  ( \ee  - \chi _1  ,   \ee  - \chi _1  )  
+ \EYttwo  ( \chi _2 , \chi _2 )  
.\end{align}
In particular, if $ (\ee  , 0) \in \dtY $, then 
$ \selfmu  = 0 $. 
\end{lem}
\begin{proof}
Applying \cite[Theorems 1,2]{o.inv2} to $ \X $, we deduce that 
$ \X $ has the scaling limit in \eqref{:35a} with the self-diffusion constant 
$ \selfmu $ given by \eqref{:42b}. This completes the proof of the first claim. 
If $ (\ee  , 0) \in \dtY $, then $ \chi = (\ee  , 0) $. 
Hence the second claim follows from the first. 
\end{proof}

In the rest of this section, 
we explain the back-ground of the representation formula \eqref{:42b}. 
To prove the convergence of $ \ve X_{\cdot / \ve ^2} $, 
we use the technique of corrector \cite{KV}, that is, 
we use a function $ \chi _{\ve }$, called corrector, for which 
\begin{align}& \notag 
\ve X_{t/\ve ^2} - \chi _{\ve }( \theta _{\ve X_{t/\ve ^2}}(\mathsf{s}))
= \text{ a continuous local martingale $ + o (\ve )$}
,\end{align}
and for which 
\begin{align}& \notag 
\limz{\ve } 
\sup_{|x| < R }
E[ |\chi _{\ve }(\theta _{\ve X_{t/\ve ^2}}(\mathsf{s}))|^2] = 0 
\quad \text{ for any } R\in\N 
,\\& \notag 
\limz{\ve } \nabla _x \chi _{\ve }(\theta_x(\mathsf{s}))|_{x=0} 
=\limz{\ve } \Dsft  \chi _{\ve }(\mathsf{s}) 
= \psi (\mathsf{s}) 
\quad \text{ in } \Lmuz 
.\end{align}
Then we have 
\begin{align}& \notag 
\limz{\ve } \ve X_{t/\ve ^2} =
\limz{\ve } \{ \ve X_{t/\ve ^2} - 
\chi _{\ve }( \theta _{\ve X_{t/\ve ^2}}(\mathsf{s}))  \} = 
M_t 
.\end{align}
Very roughly, using Fukushima decomposition 
(the Dirichlet form version of the It$ \hat{\mathrm{o}}$-Tanaka formula), 
we see that the quadratic variation of the limit martingale $ M $ is given by 
\begin{align}\label{:43g}&
\langle M \rangle_t =  2 \EY (x-\chi , x-\chi ) t 
,\end{align}
where $ \chi$ is a formal limit $ \chi= \limz{\ve } \chi _{\ve }( \mathsf{s})$. 
In practice, $ \chi _{\ve} $ diverge as $ \ve \to 0$ and 
$ x $ is not in the domain of the Dirichlet space $ (\EY , \dY )$; nevertheless, we can still justify \eqref{:43g} from formula \eqref{:42b} 
by introducing the quotient Dirichlet form $ (\EYt , \dtY  )$ 
and regarding $ \psi $ as an element of $ \dtY $. 
In fact, $ \EY (x-\chi , x-\chi )$ can be replaced by 
$ \EYtone  ( \ee  - \chi _1  ,   \ee  - \chi _1  )  
+ \EYttwo  ( \chi _2 , \chi _2)  $, where $ \chi $ corresponds to 
$ (\chi _1 , \chi _2)$. 

To apply Fukushima decomposition to 
the function $ x - \chi_{\ve }(\theta_x (\mathsf{s}))$, we use the coupled Dirichlet form $ (\EXY ,\dXY )$ because $ x - \chi_{\ve }(\theta_x (\mathsf{s}))$ belongs to $ \dXY $ locally. 
We remark that $ x - \chi_{\ve }(\theta_x (\mathsf{s}))$ cannot be in the domain of $ (\EY ,\dY )$ even locally. 

\section{Vanishing self-diffusion constant in one dimension}  \label{s:5}

In this section, we prove that $ \selfmu = 0 $ for $ d=1$ and $ \mu $ 
satisfying \eqref{:20e}. 
\begin{thm} \label{l:50} 
Suppose that $ d=1$, and assume \0. 
Then the self-diffusion constant $ \selfmu $ vanishes. 
\end{thm}

\begin{lem} \label{l:51} 
Let $ \mathcal{N}_2 = \{ \mathsf{s}\, ;\, \mathsf{s}(\{0\}) \ge 1\}$, 
and let $ \mathrm{Cap}_{\Y }$ be the capacity of the Dirichlet space 
$  (\EY , \dY ) $ on $ \Lmuz $. 
Then we obtain 
\begin{align}\label{:51a}&
\mathrm{Cap}_{\Y } (\mathcal{N}_1\cup \mathcal{N}_2 ) = 0 
.\end{align} 
\end{lem}
\begin{proof}
Recall that $ \mathrm{Cap}(\mathcal{N}_1 ) = 0 $ by assumption, which 
implies the non-collision property of $ \mathsf{X}$. 
Because 
$ \Y =\sum_{\iZs } \delta_{Y^{i}_{t}} = 
\sum_{\iZs } \delta_{X_{t}^{i}-X_{t}^{0}}$ 
is given by \eqref{:30y} and \eqref{:30z}, 
$ \mathsf{Y}$ inherits the non-collision property from $ \mathsf{X}$. 
Hence we deduce 
$ \mathrm{Cap}_{\Y } (\mathcal{N}_1) = 
\mathrm{Cap}_{\Y } (\mathcal{N}_2 ) =0 
$
from $ \mathrm{Cap}(\mathcal{N}_1 ) = 0 $. This implies \eqref{:51a}. 
\end{proof}

Let $ \varphi _N $ be the function on $ \mathsf{S}$ such that 
\begin{align}
&\notag 
\varphi _N (\mathsf{s}) = \frac{1}{N} \{ s_1 + \cdots + s_N \}
,\end{align}
where we write $ \mathsf{s}=\sum_{\iZs } \delta_{s_i}$ in such a way that 
$ s_{-1}< 0 < s_1 < s_2 < s_3 ,\ldots$. 
We note that $ \varphi _N $ is neither continuous on $ \mathcal{N}_1 $ nor 
 smooth on $ \mathcal{N}_1\cup \mathcal{N}_2  $; nevertheless, 
$ \varphi _N $ is an element of the domain of the Dirichlet form. 
Indeed, \lref{l:51} implies the following. 
\begin{lem} \label{l:52} For each $ N \in\N $, 
\begin{align}\label{:52a}&
\varphi _N \in \dY 
.\end{align}
Furthermore, $ \{ \varphi _N  \}_{N\in\N} $ is 
an $ \EY $-Cauchy sequence.  
\end{lem}

\begin{proof}
 By \lref{l:51} we have 
$ \mathrm{Cap}_{\Y }(\mathcal{N}_1\cup\mathcal{N}_2 )=0$.  
The points of discontinuity of $ \varphi $ are included in 
$ \mathcal{N}_1\cup\mathcal{N}_2  $. 
Away from the points of discontinuity, $ \varphi $ satisfies 
\begin{align}\notag & 
\DY [\varphi _N, \varphi _N] \le \frac{1}{2}\{1+  \frac{1}{N} \} 
.\end{align}
Combining these, we see that 
\begin{align}\notag & 
\EY (\varphi _N, \varphi _N) \le \frac{1}{2}\{1+  \frac{1}{N} \} 
.\end{align}
Hence we obtain claim \eqref{:52a}. 

A straightforward calculation shows that 
$ \Dsft [\varphi _M-\varphi _N, \varphi _M-\varphi _N]=0$. 
Recall that 
$ \DY  =  \Dsft  +  \, \DDD $ by \eqref{:31y}. 
Then, 
\begin{align}\notag  
\DY [\varphi _M-\varphi _N, \varphi _M-\varphi _N] 
 = &
\DDD [\varphi _M-\varphi _N, \varphi _M-\varphi _N] 
\\ \notag  \le  &
2\big\{ \DDD [\varphi _M, \varphi _M] + \DDD [\varphi _N, \varphi _N] \big\} 
= \big\{ \frac{1}{M} + \frac{1}{N}\big\} 
.\end{align} 
Thus we have
\begin{align}\notag & 
\limi{M,N} \EY (\varphi _M-\varphi _N, \varphi _M-\varphi _N)
= 0
.\end{align}
This completes the proof of \lref{l:52}. 
\end{proof}

\begin{lem} \label{l:53} 
Let $ \tilde{\varphi }_N$ be the element of $ \dtY $ whose representative 
is $ \varphi _N $. Then 
\begin{align}\label{:53a}&
\limi{N} \tilde{\varphi }_N = (1,0) 
\quad \text{ in } \EYt  
.\end{align}
In particular,  $ (1,0)\in \dtY $. 
\end{lem}
\begin{proof}
By \lref{l:52}, $ \{ \varphi _N  \}_{N\in\N} $ is an $ \EY $-Cauchy sequence. 
Hence 
we easily deduce that $ \{  \tilde{\varphi }_N \}_{N} $ 
is a Cauchy sequence in the quotient Dirichlet space $ (\EYt , \dtY )$. 

By a direct calculation 
we see that, for all $ N $ and $ \muz $-a.s.\! $ \mathsf{s}$, 
\begin{align}\notag & 
 \Dsft \varphi _N (\mathsf{s}) = 1 ,\quad  
\DDD [\varphi _N,\varphi _N]  (\mathsf{s}) = 1/2N
.\end{align}Hence we have 
$ \EYttwo (\varphi _N,\varphi _N) = 1/2N $. 
Combining these we obtain \eqref{:53a}. 
\end{proof}

\noindent {\em Proof of \tref{l:50}}  
From \lref{l:53}, we see that  $ (1,0)\in \dtY $. 
Hence we obtain $ \selfmu  = 0  $ from \lref{l:42}.  
\qed

\section{Proof of \tref{l:21}}  \label{s:6}

\tref{l:21} follows immediately from \tref{l:37} and \tref{l:50}.

\bigskip 
\noindent {List of corrections of typos: }

\noindent 
$ \bullet $ 264p 8 line from below: 
$ \sqrt{\hat{a} (\mathsf{s})} B $ $ \Rightarrow $ $ \sqrt{\hat{a}(\sx )} B $

\noindent 
$ \bullet $ 264p 6 line from below: 
$ \sqrt{\hat{a} (\mathsf{s})} B $ $ \Rightarrow $ $ \sqrt{\hat{a}(\sx )} B $

\noindent 
$ \bullet $ 265p 2 line: 
$ \sqrt{\hat{a} (\mathsf{s})} B $ $ \Rightarrow $ $ \sqrt{\hat{a}(\sx )} B $

\noindent 
$ \bullet $ 265p 3 line: 
$ \sqrt{\hat{a} (\mathsf{s})} B $ $ \Rightarrow $ $ \sqrt{\hat{a}(\sx )} B $

\noindent 
$ \bullet $ 265p 4 line: 
$ \sqrt{\hat{a} (\mathsf{s})} B $ $ \Rightarrow $ $ \sqrt{\hat{a}(\sx )} B $

\noindent 
$ \bullet $ 265p 9 line: 
$ \sqrt{\hat{a} (\mathsf{s})} B $ $ \Rightarrow $ $ \sqrt{\hat{a}(\si )} B $

\noindent 
$ \bullet $ 265p 11 line: 
$ \sqrt{\hat{a} (\mathsf{s})} B $ $ \Rightarrow $ $ \sqrt{\hat{a}(\si )} B $

\noindent 
$ \bullet $ 265p 4 line from below: 
$ \sqrt{\hat{a} (\mathsf{s})} B $ $ \Rightarrow $ $ \sqrt{\hat{a}(\si )} B $

\noindent 
$ \bullet $ 265p 3 line from below: 
$ \sqrt{\hat{a} (\mathsf{s})} B $ $ \Rightarrow $ $ \sqrt{\hat{a}(\sx )} B $

\noindent 
$ \bullet $ 268p 12 line: \\
$   \selfmu  = \EYtone  ( \ee  - \chi _1  ,   \ee  - \chi _1  )  
+ \EYttwo  ( \chi _2 , \chi _2 )  $ 
$ \Rightarrow $ 
$  \half  \selfmu  = \EYtone  ( \ee  - \chi _1  ,   \ee  - \chi _1  )  
+ \EYttwo  ( \chi _2 , \chi _2 )  $

\noindent 
$ \bullet $ 269p 4 line: $ \langle M \rangle_t =   \EY (x-\chi , x-\chi ) t $
 $ \Rightarrow $ 
$ \langle M \rangle_t =  2 \EY (x-\chi , x-\chi ) t $

\end{document}